\documentclass[10pt,a4paper]{article}
\usepackage[utf8]{inputenc}
\usepackage{amsmath}
\usepackage{amsfonts}
\usepackage{amssymb}
\usepackage{amsthm}
\usepackage{graphicx}
\usepackage[left=2cm,right=2cm,top=2cm,bottom=2cm]{geometry}
\usepackage{mathrsfs}
\usepackage{enumerate}
\usepackage{hyperref}

\newtheorem{theorem}{Theorem}[section]
\newtheorem{lemma}[theorem]{Lemma}
\newtheorem{cor}[theorem]{Corollary}
\newtheorem{prop}[theorem]{Proposition}

\begin{document}
\begin{center}
\Large{On $(k,l,H)$-kernels by walks and the $H$-class digraph}
\end{center}

\begin{center}
Hortensia Galeana-Sánchez and Miguel Tecpa-Galván
\end{center}
\[
\]

\begin{abstract}
Let $H$ be a digraph possibly with loops and $D$ a digraph without loops whose arcs are colored with the vertices of $H$ ($D$ is said to be an $H-$colored digraph). If  $W=(x_{0},\ldots,x_{n})$ is an open walk in $D$ and $i\in \{1,\ldots,n-1\}$, we say that there is an obstruction on $x_{i}$ whenever $(color(x_{i-1},x_{i}),color(x_{i},x_{i+1}))\notin A(H)$.
 
A $(k,l,H)$-kernel by walks in an $H$-colored digraph $D$ ($k\geq 2$, $l\geq 1$), is a subset $S$ of vertices of $D$, such that, for every pair of different vertices in $S$, every walk between them has at least $k-1$ obstructions, and for every $x\in V(D)\setminus S$ there exists an $xS$-walk with at most $l-1$ obstructions. This concept generalizes the concepts of kernel, $(k,l)$-kernel, kernel by monochromatic paths, and kernel by $H$-walks. If $D$ is an $H$-colored digraph, an $H$-class partition is a partition $\mathscr{F}$ of $A(D)$ such that, for every $\{(u,v),(v,w)\}\subseteq A(D)$, $(color(u,v),color(v,w))\in A(H)$ if and only if there exists $F$ in $\mathscr{F}$ such that $\{(u,v),(v,w)\}\subseteq F$. The $H$-class digraph relative to $\mathscr{F}$, denoted by $C_{\mathscr{F}}(D)$, is the digraph such that $V(C_{\mathscr{F}}(D))=\mathscr{F}$, and $(F,G)\in A(C_{\mathscr{F}}(D))$ if and only if there exist $(u,v)\in F$ and $(v,w)\in G$ with $\{u,v,w\}\subseteq V(D)$.

We will show sufficient conditions on $\mathscr{F}$ and $C_{\mathscr{F}}(D)$ to guarantee the existence of $(k,l,H)$-kernels by walks in $H$-colored digraphs, and we will show that some conditions are tight. For instance, we will show that if an $H$-colored digraph $D$ has an $H$-class partition in which every class induces a strongly connected digraph, and has an obstruction-free vertex, then for every $k\geq 2$, $D$ has a $(k,k-1,H)$-kernel by walks. Despite the fact that finding $(k,l)$-kernels in arbitrary $H$-colored digraphs is an NP-complete problem, some hypothesis presented in this paper can be verified in polynomial time.
\end{abstract}
\[
\]

$(k,l)$-kernel, $H$-colored digraph, $H$-kernel,  $(k,l,H)-$kernel by walks.

MSC class: 05C15, 05C20, 05C69.

\section{Introduction.}
Several practical and theoretical problems are related with the idea of convenient and inconvenient changes. For example, consider a group of participants, say $P$, and a set of imputations. We say that an imputation $z$ is \textit{superior} to an imputation $x$ if there is an influence-group into $P$, say $S_{xz}$, that can convince a sufficient number of participants that imputation $z$ has more benefits for $P$ than $x$, and such change gives particular benefits to $S_{xz}$. It is worth to find a set of imputations that represents \textit{good choices} for $P$. 

A digraph $D$ can be defined as follows: the vertex set is the set of imputations, and $(x,z)$ is an arc in $D$ if and only if there exists an influence-group $S_{xz}$ that can convince enough participants in $P$ to choose imputation $z$ rather than $x$, and such choice gives a particular benefit to the members in $S_{xz}$. Notice that the existence of a path $(x_{0}, x_{1}, \ldots , x_{n})$ in $D$ implies successive improvements for the participants. However, the change from $x_{1}$ to $x_{2}$ does not necessarily imply a particular benefit for $S_{x_{0}x_{1}}$, and conflicts of interest may arise that prevent the development of improvements for $P$.

Hence, we need to consider the possible conflict of interest that may arise through a chain of improvements. First, let $\rho$ be the arc-coloring in $D$ that assigns to an arc $(x,z)$ the influence-group $S_{xz}$. 
Now consider a secondary digraph $H$ whose vertices are the different influence-groups represented in the arcs of $D$, and $(G, G')$ is an arc in $H$ whenever a benefit to $G'$ implies a benefit to $G$. Notice that the arcs of the digraph $D$ are colored with the vertices of $H$, hence, $D$ is called an $H$-colored digraph.
Given two consecutive arcs in $D$, say $(x,z)$ and $(z,w)$, we say that there is a \textit{convenient change of color in $z$} if and only if $(\rho (x,z), \rho (z,w)) \in A(H)$. In this case, the influence-groups $S_{xz}$ and $S_{zw}$ will not be in a conflict of interest, and the improvement from $x$ to $z$, and then $z$ to $w$ will be possible. In the same way, an inconvenient change of colors implies a conflict of interest between $S_{xz}$ and $S_{zw}$, and establish agreements between both groups will be necessary in order to achieve an improvement from $x$ to $w$. 

In this sense, one may think that an imputation $z$ is more desirable than imputation $x$ if there exists an $xz$-walk in $D$ with no inconvenient changes of color. However, such kind of walks may not necessarily exist, and it is important to consider the number of inconvenient changes of color to reach $z$ from $x$. Given $k \geq 2$, and two different imputations $x$ and $z$, we will say that $z$ is \textit{more desirable} than $x$ if there exists an $xz$-walk in $D$ with less than $k-1$ inconvenient changes of color. A $(k,H)$-kernel by walks in $D$ is a set of imputations such that: (i) for every $x \in S$, there is no $z \in S$ such that $z$ is more desirable than $x$, and (ii) for every $x \in V(D) \setminus S$, there exists $z \in S$ such that $z$ is more desirable than $x$. A $(k,H)$-kernel by walks in $D$ represents a good choice of imputations for $P$. 

The formal notion of $(k,H)$-kernel was introduced in \cite{nearly}, and it is worth mentioning that the authors worked with paths, instead of walks, with at most certain number of inconvenient changes of color. 
Although there exist certain classes of $H$-colored digraphs that have a $(k,H)$-kernel for certain values of $k$, it is known that determine if an arbitrary $H$-colored digraph has a $(k,H)$-kernel is an NP-complete problem. Hence, it is worth to find conditions that guarantee the existence of such kind of kernels. However, there are several parameters to consider in order to find $(k,H)$-kernels in arbitrary $H$-colored digraphs, as (i) the digraph $H$, (ii) the $H$-colored digraph, (iii) the coloring of the arcs of $D$, and (iv) the value of $k$. In \cite{nearly} the authors worked by considering certain properties of the $H$-colored digraph.

In this paper we will present some sufficient conditions for the existence of a more general kind of kernels, namely $(k,l,H)$-kernels by walks, by 
considering two ideas: the $H$-colored digraph $D$ has a nice coloring on its arcs, and find $(k,l,H)$-kernels by walks in $D$ through the existence of $(k,l)$-kernels in an auxiliary digraph.

The way we will approach this problem is by means of certain kind of partitions of the arcs of an $H$-colored digraph, which will be called $H$-class partitions, and an auxiliary digraph, called the $H$-class digraph. The $H$-class partitions and the $H$-class digraph simplify the behavior of the arc-coloring in the $H$-colored digraph. Moreover, under certain conditions of such structures, we will be able to guarantee the existence of $(k,H)$-kernels in $H$-colored digraphs. For instance: 

\begin{theorem}(Theorem \ref{teostrong2})
If $D$ is an $H$-colored digraph and $\mathscr{F}$ is an $H$-class partition of $D$ such that every class in $\mathscr{F}$ induces a strongly connected digraph in $D$ and has an obstruction-free vertex in $D$, then for every $k \geq 2$, $D$ has a $(k,H)$-kernel by walks.
\end{theorem}

\begin{theorem}(Theorem \ref{corunilat})
Let $D$ be an $H$-colored digraph, $\mathscr{F}$ a walk-preservative $H$-class partition of $A(D)$ such that for every $F \in \mathscr{F}$, $D\langle F \rangle$ is unilateral and has no sinks. If  $C_{\mathscr{F}}(D)$ has a $(k,l)$-kernel for some $k \geq 3$ and $l \geq 1$, then $D$ has a $(k-1,l+1,H)$-kernel by walks.
\end{theorem}

Some hypothesis presented in this paper can be verified in polynomial time, and under the hypothesis presented in this paper, find a $(k,l,H)$-kernel is polynomial solvable. Moreover, we will show that some conditions are tight. 

\section{Definitions and previous results}

For terminology and notation not defined here, we refer the reader to  \cite{1}.  If $D =(V(D), A(D))$ is a digraph and $x\in V(D)$, we denote by $A^{+}(x)$ the set $\{ (x,v) \in A(D): v \in V(D) \}$, $A^{-}(x)$ the set $\{ (u,x) \in A(D) : u \in V(D) \}$, and $A(x)=A^{-}(x) \cup A^{+}(x)$. If $D$ is a digraph without loops, a \emph{sink} is a vertex $x$ such that $A^{+}(x)=\emptyset$.
If $F$ is a subset of arcs in $D$, we denote by $D\langle F \rangle$ the subdigraph \textit{arc-induced by $F$}, that is $A(D\langle F \rangle) = F$ and $V(D\langle F \rangle)$ consist in those vertices of $D$ which are incident with at least one arc in $F$.

If $H$ is a digraph possibly with loops, a \emph{sink} is a vertex $x$ such that $A^{+}(x) \subseteq \{(x,x)\}$.  If  $S_1$ is a subsets of $V(D)$, we denote by $N^{+}(S_{1})$ the \emph{proper out-neighbor of $S_{1}$}.  In this paper we write walk, path and cycle, instead of directed walk, directed path, and directed cycle, respectively.
 If $W=(x_{0}, \ldots , x_{n})$ is a walk (path), we say that $W$ is an $x_0x_n$-\emph{walk} ($x_0x_n$-\emph{path}). The \emph{length} of $W$ is the number $n$ and it is denoted by $l(W)$.
If $T_{1}=(z_{0}, \ldots , z_{n})$ and $T_{2}=(w_{0}, \ldots , w_{m})$ are walks and $z_{n}=w_{0}$, we denote by $T_{1} \cup T_{2}$ the walk $(z_{0}, \ldots , z_{n} = w_{0}, \ldots , w_{m})$. 
If $x$ belongs to a walk $W$, we denote by $x^{+}$ (respectively $x^{-}$) the successor (respectively predecessor) of $x$ in $W$.

If $S_1$ and $S_2$ are two disjoint subsets of  $V(D)$, a $uv$-walk  in  $D$ is called an $S_1S_2$-\emph{walk} whenever $u$ $\in$ $S_1$ and $v$ $\in$ $S_2$. If $S_1 = \{x\}$ or $S_2 = \{x\}$, then we write $xS_2$-\emph{walk} or $S_1x$-\emph{walk}, respectively. 
A digraph $D$ is \textit{unilateral} if for every $\{u,v\} \subseteq V(D)$ there exists either a $uv$-path or a $vu$-path. A \textit{strongly connected} digraph is a digraph such that for every $\{u,v\} \subseteq V(D)$, there exist a $uv$-path and a $vu$-path. 

The concept of \emph{kernel} was introduced by von Neumann and Morgenstern in \cite{2} as a subset $S$ of vertices of a digraph $D$, such that for every pair of different vertices in $S$, there is no arc between them, and every vertex not in $S$ has at least one out-neighbor in $S$. This concept has been deeply and widely studied by several authors, for example \cite{ker1}, \cite{ker2}, \cite{ker3} and \cite{ker4}. In \cite{3} Chvátal showed that deciding if a digraph has a kernel is an NP-complete problem, and a classical result proved by König \cite{kon} shows that every transitive digraph has a kernel.

A subset $S$ of vertices of $D$ is said to be a \emph{kernel by paths}, if for every $x \in V(D)\setminus S$, there exists an $xS$-path (that is, $S$ is \emph{absorbent by paths}) and, for every pair of different vertices $\{ u, v \} \subseteq S$, there is no $uv$-path in $D$ (that is, $S$ is \emph{independent by paths}). 
This concept was introduced by Berge in \cite{19}, and it is a well-known result that every digraph has a kernel by paths \cite{19} (see Corollary 2 on p. 311). Moreover, such kind of kernels can be constructed by taking one arbitrary vertex in each terminal strong component of the digraph. The following lemma will be useful.

\begin{lemma}
\label{c0.l1}
If $D$ is a digraph with not isolated vertices and $K$ is a kernel by paths in $D$, then for every $x \in K$, $d^{-}_{D}(x) \neq 0$.
\end{lemma}
\begin{proof}
Proceeding by contradiction, suppose that there exists $x \in K$ such that $d^{-}_{D}(x) = 0$. Since $D$ has not isolated vertices, then there exists $y \in V(D)$ such that $(x,y) \in A(D)$, which implies that $y \notin K$. Hence, there exists a $yz$-path in $D$, say $P$, such that $z \in K$. Since $d^{-}_{D}(x)=0$, we have that $z \neq x$, which implies that $(x,y) \cup P$ is an $xz$-path in $D$ with $\{ x,z \} \subseteq K$, contradicting the independence by paths of $K$. Therefore, $d^{-}_{D}(x) \neq 0$.
\end{proof}

The concept of $(k,l)$-kernel was introduced by Borowiecki and Kwa\'snik in \cite{14} as follows:
If $k \geq 2$, a subset $S$ of vertices of a digraph $D$ is a \emph{$k$-independent set}, if for every pair of different vertices in $S$, every walk between them has length at least $k$. If $l \geq 1$, we say that $S$ is an \emph{$l$-absorbent set} if for every $x \in V(D) \setminus S$ there exists an $xS$-walk with length at most $l$.
If $k \geq 2$ and $l \geq 1 $, a \emph{$(k,l)$-kernel} is a subset of $V(D)$ which is $k$-independent and $l$-absorbent. If $l=k-1$, the $(k,l)$-kernel is called a \emph{$k$-kernel}. Notice that every $2$-kernel is a kernel. Sufficient conditions for the existence of $k$-kernels have been proved, for example see \cite{symmetricdigraphs}, \cite{6} and \cite{5}. In \cite{symmetricdigraphs} the authors proved the following theorem:
  
\begin{theorem}\cite{symmetricdigraphs}
\label{c0.t1}
If $D$ is a symmetric digraph, then $D$ has a $k$-kernel for every $k \geq 2$. Moreover, every maximal $k$-independent set in $D$ is a $k$-kernel.
\end{theorem}

As a consequence, we have the following lemma.

\begin{lemma}
\label{symetricklkernel} Let $D$ be a symmetric digraph and $\{k, l \} \subseteq \mathbb{N}$. If $2 \leq k $ and $k-1 \leq l$, then $D$ has a $(k,l)$-kernel.
\end{lemma}
\begin{proof}
It follows from Theorem \ref{c0.t1} and the definition of $(k,l)$-kernel.
\end{proof}

A digraph is \emph{$m$-colored} if its arcs are colored with $m$ colors. If $D$ is an $m$-colored digraph, a path in $D$ is called \emph{monochromatic} (respectively, \emph{alternating}) if all of its arcs are colored alike (respectively, consecutive arcs have different colors). 
A subset $S$ of vertices of $D$ is a \emph{kernel by monochromatic paths} (respectively, \emph{kernel by alternating walks}) if for every $x \in V(D) \setminus S$ there exists a monochromatic $xS$-path (respectively an alternating $xS$-walk), and no two different vertices in $S$ are connected by a monochromatic path (respectively, by an alternating walk). 
Notice that a digraph $D$ has a kernel if and only if the $m$-colored digraph $D$, in which every two different arcs have different colors, has a kernel by monochromatic paths. 
The existence of kernels by monochromatic paths was introduced in \cite{9}. 
Due to the difficulty of finding kernels by monochromatic paths in $m$-colored digraphs, in \cite{11} was defined the \emph{color-class digraph} of an $m$-colored digraph, denoted by $\mathscr{C}(D)$, as the digraph whose set of vertices are the colors represented in the arcs of $D$, and $(c_{1}, c_{2})$ is an arc in $\mathscr{C}(D)$ if and only if there exist two arcs of $D$, namely $(u,v)$ and $(v,w)$, such that $(u,v)$ has color $c_{1}$ and $(v,w)$ has color $c_{2}$. Several conditions on the color-class digraph guarantee the existence of a kernel by monochromatic paths.

Let $H$ be a digraph possibly with loops, and $D$ a digraph without loops whose arcs are colored with the vertices of $H$ ($D$ is said to be an $H$-colored digraph). For an arc $(x,z)$ of $D$, we denote by $\rho(x,z)$ its color. A vertex $x \in V(D)$ is \emph{obstruction-free in $D$} if $(\rho (a), \rho (b)) \in A(H)$ whenever $a \in A^{-}(x)$ and $b \in A^{+}(x)$. We say that a subdigraph $D'$ of $D$ is an \emph{$H$-digraph}, if for every two arcs $(u,v)$ and $(v,w)$ in $D'$ we have that $(\rho (u,v), \rho (v,w))\in A(H)$. 

An \emph{$H$-class partition of $A(D)$} is a partition of $A(D)$, say $\mathscr{F}$, such that for every $\{ (u,v), (v,w)\} \subseteq A(D)$, $(\rho (u,v), \rho (v,w)) \in A(H)$ if and only if there exists $F$ in $\mathscr{F}$ such that $\{ (u,v),(v,w)\} \subseteq F$. An $H$-class partition $\mathscr{F}$ is \emph{walk-preservative} if  for every $(F,G) \in A(C_{\mathscr{F}}(D))$ and $z \in V(D\langle F \rangle)$, there exists a $zw$-path in $D\langle F \rangle$ for some $w \in  V(D \langle G \rangle )$. Notice that $w \in  V(D \langle F \rangle ) \cap V(D \langle G \rangle )$. If $x \in V(D)$, we define $N^{-}_{\mathscr{F}}(x)=\{ F \in \mathscr{F} : (u,x) \in F$ for some $u \in V(D)\}$, $N^{+}_{\mathscr{F}}(x)=\{ F \in \mathscr{F} : (x,v) \in F$ for some $v \in V(D)\}$, and $N_{\mathscr{F}}(x)=N^{+}_{\mathscr{F}}(x) \cup N^{-}_{\mathscr{F}}(x)$. 

If $\mathscr{F}$ is an $H$-class partition of $A(D)$, the \emph{$H$-class digraph relative to $\mathscr{F}$}, denoted by $C_{\mathscr{F}}(D)$, is the digraph such that $V(C_{\mathscr{F}}(D)) = \mathscr{F}$, and $(F_{i}, F_{j})$ is an arc in $C_{\mathscr{F}}(D)$, if and only if there exist $(u,v) \in F_{i}$ and $(v,w) \in F_{j}$ for some $\{ u, v, w \} \subseteq V(D)$. 
Notice that $C_{\mathscr{F}}(D)$ can allow loops. Moreover, $\mathscr{C}(D)$ is a particular case of $C_{\mathscr{F}}(D)$ when $H$ has only loops, every vertex in $H$ has a loop and every class in $\mathscr{F}$ consist in those arcs colored alike. 

If $W=(x_{0}, \ldots , x_{n})$ is a walk in an $H$-colored digraph $D$, and $i \in \{0, \ldots , n-1 \}$, we say that there is an \textit{obstruction on $x_{i}$} if and only if $(\rho (x_{i-1}, x_{i}), \rho (x_{i}, x_{i+1})) \notin A(H)$ (indices are taken modulo $n$ if $x_{0}=x_{n}$). 
We denote by $O_{H}(W)$ the set $\{ i \in \{ 1, \ldots , n-1\} : \text{there is an obstruction on }x_{i} \}$ ($O_{H}(W)=\{ i \in \{ 0, \ldots , n-1\} : \text{there is an obstruction on }x_{i} \}$ if $W$ is closed).
A walk without obstructions is called \textit{$H$-walk.}
If $D$ is an $H$-colored digraph and $W$ is a walk in $D$, the \textit{$H$-length of $W$}, denoted by $l_{H}(W)$, is defined as either $l_{H}(W)= |O_{H}(W)|+1$ if $W$ is open or  $l_{H}(W)= |O_{H}(W)|$ otherwise. 
Notice that the usual length is a particular case of the $H$-length when $H$ has no arcs nor loops. The $H$-length was studied in \cite{12} for closed walks and in \cite{20} for open walks.

If $l \geq 1$, a subset $S$ of vertices of $D$ is an \textit{$(l, H)$-absorbent set by walks} if for every vertex $v \in V(D) \setminus S$ there exists a $vS$-walk whose $H$-length is at most $l$. If $k \geq 2$, we say that $S$ is a \textit{$(k,H)$-independent set by walks} if for every pair of different vertices in $S$, every walk between them has $H$-length at least $k$.
If $k\geq 2$ and $l\geq 1$, we say that $S$ is a \textit{$(k, l,H)$-kernel by walks} if it is both $(k,H)$-independent by walks and $(l,H)$-absorbent by walks. If $l=k-1$, a $(k,l,H)$-kernel by walks is called a \textit{$(k, H)$-kernel by walks}. It is straightforward to see the following lemma.

 \begin{lemma}
\label{hdigraph}
If $D$ is an $H-$digraph, then every kernel by paths in $D$ is a $(k,l,H)$-kernel by walks in $D$ for every $k \geq 2$ and $l \geq 1$.
\end{lemma}

It is worth mentioning that the concepts of $(k,l,H)$-kernel by paths (introduced in \cite{nearly}) and $(k,l,H)$-kernel by walks are not equivalent. In \cite{ger}, the authors showed an infinite family of digraphs with $(2,H)$-kernel by walks and no $(2,H)$-kernel by paths, and an infinite family of digraphs with $(2,H)$-kernel by paths and no $(2,H)$-kernel by walks.

Some kinds of kernels are particular cases of the $(k,l,H)$-kernels by walks in $H$-colored digraphs. For instance, if $S$ is a $(k,l,H)$-kernel by walks in an $H$-colored digraph, then: 
(i) $S$ is a kernel (von Neumann and Morgenstern \cite{2}) if $k=2$, $l=1$ and $H$ has no arcs nor loops,
(ii) $S$ is a $(k,l)$-kernel (Borowiecki and Kwa\'snik in \cite{14}) if
$H$ is a digraph without arcs nor loops, 
(iii) $S$ is a kernel by monochromatic paths (Sands, Sauer and Woodrow \cite{9}) if  $k=2$, $l=1$ and $H$ is a looped digraph and those are the only arcs in $H$, and 
(vi) $S$ is an $H$-kernel  (Arpin and Linek \cite{7}) if $k=2$ and $l=1$.

Finally, the following lemma will be useful in what follows.

\begin{lemma}
\label{noisolatedvertices}
Let $D$ be a digraph with at least one arc, $W=\{ x \in V(D): d(x)=0\}$, and $\{ k, l \} \subseteq \mathbb{N}$ such that $2\leq k$ and $1\leq l$. If $K$ is a $(k,l, H)$-kernel by walks in $D - W$, then $K \cup W$ is a $(k,l,H)$-kernel by walks in $D$.
\end{lemma}
 
\section{First results}

In this section we will show several technical lemmas and corollaries in order to simplify the proofs of the main results.

\begin{lemma}
\label{obs1}
Let $D$ be an $H$-colored digraph, $\mathscr{F}$ an $H$-class partition of $A(D)$ and $x \in V(D)$. The following assertions hold:
	\begin{enumerate}[a)]
	\item If $d^{-}(x) \neq 0$ and $d^{+}(x) \neq 0$, then for every $F_{1} \in N^{-}_{\mathscr{F}}(x)$ and $F_{2} \in N^{+}_{\mathscr{F}}(x)$, we have that $(F_{1}, F_{2}) \in A(C_{\mathscr{F}}(D))$. 
	
	\item If $x$ is obstruction-free in $D$ and $d(x) \neq 0$, then there is a unique $F \in \mathscr{F}$ such that $x \in V(D \langle F \rangle )$. 
	
	\item If $T$ is a walk in $D$ such that $O_{H}(T)=\emptyset$, then there is a unique $F \in \mathscr{F}$ such that $A(T) \subseteq F$.
	
	\item If $u$, $v$ and $w$ are three different vertices in $D$, $T$ is a $uv-$walk, and $T'$ is a $vw$-$H$-walk, then either $l_{H}(T \cup T')=l_{H}(T)$ or $l_{H}(T \cup T') = l_{H}(T) +1$.
	\end{enumerate}
\end{lemma}
\begin{proof}
\begin{enumerate}[a)]
\item If $F_{1} \in N^{-}_{\mathscr{F}}(x)$ and $F_{2} \in N^{+}_{\mathscr{F}}(x)$, then there exists $\{u, v \} \subseteq V(D)$ such that $(u,x) \in F_{1}$ and $(x,v) \in F_{2}$. It follows from the definition of $C_{\mathscr{F}}(D)$ that $(F_{1}, F_{2}) \in A(C_{\mathscr{F}}(D))$.

\item Since $d(x) \neq 0$, then there exists $F \in \mathscr{F}$ such that $x \in V(D\langle F \rangle )$. On the other hand, since $x$ is obstruction-free, we have that $A(x) \subseteq F$, concluding that $F$ is unique.

\item If $T=(x_{0}, \ldots , x_{n})$, then $(\rho (x_{i-1}, x_{i}), \rho (x_{i}, x_{i+1})) \in A(H)$ for every $i \in \{0, \ldots , n-1\}$ (indices modulo $n$ if $x_{0}=x_{n}$). Hence, it follows from the definition of $H$-class partition that there is a unique $F \in \mathscr{F}$ such that $A(T) \subseteq F$.

\item Suppose that $T=(z_{0} = u, z_{1}, \ldots , z_{n}=v)$ and $T'=(z_{n} = v, z_{n+1}, \ldots , z_{m} =w)$. It is straightforward to see that $O_{H}(T \cup T') \subseteq  O_{H}(T) \cup \{ n \}$ and $O_{H}(T) \subseteq O_{H}(T \cup T')$, which implies that either $l_{H}(T \cup T')=l_{H}(T)$ or $l_{H}(T \cup T') = l_{H}(T) +1$.
\end{enumerate}
\end{proof}

\begin{lemma}
\label{c1.l1}
Let $D$ be an $H$-colored digraph and $\mathscr{F}$ an $H$-class partition of $A(D)$. If $\mathcal{S}$ is an independent set in $C_{\mathscr{F}}(D)$, then $D\langle \cup _{F \in  \mathcal{S}} F\rangle$ is an $H$-subdigraph of $D$.
\end{lemma}
\begin{proof}
Let $D'=D\langle \cup _{F \in  \mathcal{S}} F\rangle$ and  $\{ (u,v), (v,x)\} \subseteq A(D')$. We will prove that $ ( \rho (u,v), \rho (v,x)) \in A(H)$. 
By definition of $D'$, there exists $\{ F,G\} \subseteq \mathcal{S}$ such that $(u,v) \in F$ and $(v,x) \in G$. Hence, we have that $(F,G) \in A(C_{\mathscr{F}}(D))$. Since $\mathcal{S}$ is an independent set in $C_{\mathscr{F}}(D)$, then $F=G$, which implies that $\{ (u,v), (v,x) \} \subseteq F$.
It follows from the fact that $\mathscr{F}$ is an $H$-class partition of $A(D)$ that $(\rho (u,v) , \rho (v,x)) \in A(H)$. Therefore, $D'$ is an $H$-subdigraph of $D$.
\end{proof}

The following lemmas will show some properties of the $H$-class partitions and the $H$-class digraph related to connectivity.

\begin{lemma}
\label{c1.l2}
Let $D$ be an $H$-colored digraph and $\mathscr{F}$ an $H$-class partition of $A(D)$. If $D$ is strongly connected, then $C_{\mathscr{F}}(D)$ is strongly connected.
\end{lemma}
\begin{proof}
Let $F$ and $F'$ be different vertices in $C_{\mathscr{F}}(D)$, and $\{ (x_{0}, x_{1} ), (z_{0},z_{1})\} \subseteq A(D)$ such that $(x_{0}, x_{1}) \in F$ and $(z_{0}, z_{1}) \in F'$. It follows from the fact that $D$ is strongly connected, that there exists an $x_{0}z_{1}$-walk in $D$, say $C=(w_{0}, \ldots , w_{n})$, whose initial arc is $(x_{0}, x_{1})$ and ending arc is $(z_{0}, z_{1})$.

Since $(w_{0}, w_{1}) \in F$, $(w_{n-1}, w_{n}) \in F'$, and $F \neq F'$, we can conclude that $O_{H}(C) \neq \emptyset$ (Lemma \ref{obs1}) (c)). Let $O_{H}(C)=\{ \alpha_{1}, \ldots , \alpha_{r} \}$ for some $r \geq 1$, and we can assume that $\alpha_{i} \leq \alpha_{i+1}$ whenever  $i \in \{ 1, \ldots , r-1 \}$. For every $i \in \{ 1, \ldots , r \}$, let $G_{i} \in \mathscr{F}$ such that $(w_{\alpha_{i}}, w_{\alpha_{i}}^{+}) \in G_{i}$. Notice that $P=(G_{1}, \ldots , G_{r})$ is a walk in $C_{\mathscr{F}}(D)$, and $G_{r}=F'$.

 If $\alpha_{1}=0$, then $F=G_{1}$, which implies that $P$ is an $FF'$-walk in $C_{\mathscr{F}}(D)$. If $\alpha_{1}\neq 0$, then $(F, G_{1})\in A(C_{\mathscr{F}}(D))$, which implies that $(F, G_{1}) \cup P$ is an $FF'$-walk in $C_{\mathscr{F}}(D)$. Hence, we conclude that $C_{\mathscr{F}}(D)$ is strongly connected.
\end{proof}

\begin{lemma}
\label{unilateral.lema}
Let $D$ be an $H$-colored digraph, $\mathscr{F}$ an $H$-class partition of $A(D)$,  $\mathcal{S}$ a non-empty subset of $V(C_{\mathscr{F}}(D))$ such that $D\langle F \rangle$ is unilateral for every $F\in \mathcal{S}$, $K$ a kernel by walks in $D\langle \cup_{F\in \mathcal{S}} F \rangle$, and $\{x,z\} \subseteq K$. If  $x \in V(D\langle F_{1} \rangle )$  and $z \in V(D \langle F_{2} \rangle)$ for some $\{F_{1}, F_{2}\} \subseteq \mathcal{S}$, then $F_{1} \neq F_{2}$.
\end{lemma}
\begin{proof}
Proceeding by contradiction, suppose that $F_{1} = F_{2}$. Since $D\langle F_{1} \rangle$ is unilateral, then either there exists an $xz$-path in $D \langle F_{1} \rangle$ or there exists a $zx$-path in $D\langle F_{1} \rangle$, say $P$. It follows that $P$ is a path in $D\langle \cup_{F\in \mathcal{S}} F \rangle$ such that $\{x,z\} \subseteq K$, which contradicts the independence by paths of $K$. Therefore, $F_{1} \neq F_{2}$.
\end{proof}

\begin{lemma}
\label{lemmastronglyconnected}
Let $D$ be an $H$-colored digraph and $\mathscr{F}$ an $H$-class partition of $A(D)$ such that for every $F \in \mathscr{F}$, $D\langle F \rangle$ is strongly connected. The following assertions hold:
\begin{enumerate}[a)]
\item $\mathscr{F}$ is walk-preservative.

\item $C_{\mathscr{F}}(D)$ is a symmetric digraph.

\item If $\mathcal{S} \subseteq V(C_{\mathscr{F}}(D) )$ is an independent set in $C_{\mathscr{F}}(D)$ and $\{F_{1}, F_{2} \} \subseteq \mathcal{S}$, then $V(D \langle F_{1} \rangle ) \cap V(D \langle F_{2} \rangle) = \emptyset.$
\end{enumerate}
\end{lemma}
\begin{proof}
a)  Let $(F, G) \in A(C_{\mathscr{F}}(D) )$ and $x \in V(D \langle F \rangle )$. It follows from the definition of $C_{\mathscr{F}}(D)$ that there exists $\{ u, v, z \} \subseteq V(D)$ such that $(u,v) \in F$ and $(v,z) \in G$. Notice that $v \in V(D\langle F \rangle ) \cap V(D \langle G \rangle )$. Since $D\langle F \rangle$ is strongly connected, then there exists an $xv$-walk in $D \langle F \rangle$, which implies that $\mathscr{F}$ is walk-preservative.

b) Let $(F_{1}, F_{2}) \in A(C_{\mathscr{F}}(D) )$. It follows from the definition of $C_{\mathscr{F}}(D)$ that there exists $\{u, v, z \}\subseteq V(D)$ such that $(u,v) \in F_{1}$ and $(v,z) \in F_{2}$. 
Since $(u,v) \in A(D\langle F_{1} \rangle )$, then  $D\langle F_{1} \rangle$ is a nontrivial strongly connected digraph, which implies that  there exists $u' \in V(D\langle F_{1} \rangle)$ such that $(v,u') \in F_{1}$. In the same way,  there exists $z' \in V(D \langle F_{2} \rangle )$ such that $(z',v) \in F_{2}$. It follows from the definition of $C_{\mathscr{F}}(D)$ that $(F_{2}, F_{1}) \in  A(C_{\mathscr{F}}(D) )$, concluding that $C_{\mathscr{F}}(D)$ is a symmetric digraph.
   
c) Proceeding by contradiction, suppose that  $V(D \langle F_{1} \rangle ) \cap V(D \langle F_{2} \rangle) \neq \emptyset$, and consider $x \in  V(D \langle F_{1} \rangle ) \cap V(D \langle F_{2} \rangle)$. 
Since $D \langle F_{1} \rangle$ and $D \langle F_{2} \rangle$ are non-trivial strongly connected digraphs, then there exist $u \in V( D \langle F_{1} \rangle )$ and $z \in V(D \langle F_{2} \rangle )$ such that $(u,x) \in A(D\langle F_{1} \rangle)$ and $(x,z) \in A(D \langle F_{2} \rangle )$. It follows from the definition of $C_{\mathscr{F}}(D)$ that $(F_{1}, F_{2}) \in A(C_{\mathscr{F}}(D) )$, which is not possible since $\mathcal{S}$ is an independent set in $C_{\mathscr{F}}(D)$. Therefore, $V(D \langle F_{1} \rangle ) \cap V(D \langle F_{2} \rangle) = \emptyset.$
\end{proof}

\begin{theorem}
\label{classlema}
If $D$ is an $H$-colored digraph such that for every vertex $x \in V(D)$, there exists an $xw$-$H$-walk for some obstruction-free vertex $w$ in $D$, then $D$ has a kernel by $H$-walks.
\end{theorem}
\begin{proof}
First, we define the digraph $D'$ whose vertex set consists of the obstruction-free vertices of $D$, and $(x,z) \in A(D')$ if and only if there exists an $xz$-$H$-walk in $D$. Now, we will show that $D'$ has a kernel, say $K$, by proving that $D'$ is a transitive digraph. Then, a simple proof will show that $K$ is a kernel by $H$-walks in $D$.

In order to show that $D'$ is a transitive digraph, consider $\{(u,v), (v,w) \} \subseteq A(D')$. It follows from the definition of $D'$ that there exists a $uv$-$H$-walk in $D$, say $W_{1}$, and a $vw$-$H$-walk in $D$, say $W_{2}$. Since $v$ is an obstruction-free vertex in $D$, then we have that $W_{1} \cup W_{2}$ is a $uw$-$H$-walk in $D$, which implies that $(u,w) \in A(D')$, concluding that $D'$ is a transitive digraph.  

Since $D'$ is a transitive digraph, consider a kernel in $D'$, say $K$. We will show that $K$ is a kernel by $H$-walks in $D$. It follows from the definition of $D'$ and the fact that $K$ is an independent set in $D'$, that $K$ is an independent set by $H$-walks in $D$. It only remains to show that $K$ is an absorbent set by $H$-walks in $D$.

 Consider $x \in V(D)\setminus K$. If $x$ is an obstruction-free vertex in $D$, then $x \in V(D')$ and, since $K$ is a kernel in $D'$, there exists $w \in K$ such that $(x,w) \in A(D')$, which implies that there exists an $xw$-$H$-walk in $D$. If $x \notin V(D')$, then by hypothesis, there exists $z \in V(D')$ and an $xz$-$H$-walk in $D$, say $W_{1}$. If $z \in K$, then $W_{1}$ is an $xK$-$H$-walk in $D$. If $z \notin K$, then there exists $w \in K$ such that $(z,w) \in A(D')$, which implies that there exists a $zw$-$H$-walk in $D$, say $W_{2}$. Since $z$ is an obstruction-free vertex in $D$, then $W_{1} \cup W_{2}$ is an $xK$-$H$-walk in $D$, concluding that $K$ is a kernel by $H$-walks in $D$.
\end{proof}

Two simple but interesting corollaries of the previous result can be shown considering particular patterns.

\begin{cor}
If $D$ is an $m$-colored digraph such that for every vertex $x \in V(D)$, there exists a monochromatic $xw$-walk with color $c$, for some vertex $w$ such that every arc in $A(w)$ has color $c$, then $D$ has a kernel by monochromatic paths.
\end{cor}
\begin{proof}
Consider the digraph $H$ whose vertices are the colors represented in $A(D)$, and $A(H)=\{ (c,c) : c \in V(H) \}$. Since $D$ is an $H$-colored digraph satisfying the hypothesis on Lemma \ref{classlema}, it follows that $D$ has a kernel by $H$-walks, which is a kernel by monochromatic paths in $D$.
\end{proof}

\begin{cor}
If $D$ is an $m$-colored digraph such that for every vertex $x \in V(D)$, there exists an alternating $xw$-walk, for some vertex $w$ such that $A^{-}(w)$ and $A^{+}(w)$ have not colors in common, then $D$ has a kernel by properly colored walks.
\end{cor}
\begin{proof}
Consider the digraph $H$ whose vertices are the colors represented in $A(D)$, and $A(H)=\{ (c,d) : \{c, d \} \subseteq V(H), c \neq d \}$.  Since $D$ is an $H$-colored digraph satisfying the hypothesis on Theorem \ref{classlema}, it follows that $D$ has a kernel by $H$-walks, which is a kernel by alternating walks in $D$.
\end{proof}

The goal of the following lemma is to show the importance of the notion of a walk-preservative $H$-class partition. Such condition allows us to find $(l+1,H)$-absorbent sets by walks in an $H$-colored digraph through $l$-absorbent sets in the $H$-class digraph. 

\begin{prop}
\label{c1.p2}
Let $D$ be an $H$-colored digraph with not isolated vertices, $\mathscr{F}$ a walk-preservative $H$-class partition of $A(D)$, and $\mathcal{S}$ an independent and $l$-absorbent set in $C_{\mathscr{F}}(D)$ for some $l \geq 1$. If $K$ is a kernel by paths in $D\langle \cup_{F\in \mathcal{S}}F \rangle$, then $K$ is an $(l+1, H)$-absorbent set by walks in $D$.
\end{prop}
\begin{proof}
Let $D'=D\langle \cup_{F\in \mathcal{S}}F \rangle$ and $x_{0} \in V(D) \setminus K$. If $x_{0} \in V(D')$, since $K$ is a kernel by paths in $D'$, there exists an $x_{0}K$-path in $D'$, say $T$. It follows from Lemma \ref{c1.l1} that $T$ is an $H$-path. Hence, $l_{H}(T)=1$, which implies that $l_{H}(T)\leq l+1$. 

Now we will assume that $x_{0} \notin V(D')$. By hypothesis, $x_{0}$ is not isolated, which implies that $x_{0} \in V(D\langle F_{0} \rangle)$ for some $F_{0} \in \mathscr{F}$. Notice that $F_{0} \notin \mathcal{S}$ because $x_{0} \notin V(D')$. Since $\mathcal{S}$ is an $l$-absorbent set in $C_{\mathscr{F}}(D)$, we can consider an $F_{0}F_{t}$-path with minimum length in $C_{\mathscr{F}}(D)$, say $T=(F_{0}, \ldots , F_{r})$, where $F_{r} \in \mathcal{S}$. Notice that $r \leq l$.

Since $\mathscr{F}$ is walk-preservative, there exists $x_{\alpha_{1}} \in V(D\langle F_{0} \rangle ) \cap V(D\langle F_{1} \rangle)$ such that there exists an $x_{0}x_{\alpha_{1}}$-path in $D\langle F_{0} \rangle$, say $T_{0}$,
and, for every $i \in \{2, \ldots , r-1 \}$, there exist $x_{\alpha_{i}} \in V(D\langle F_{i-1} \rangle ) \cap V(D\langle F_{i} \rangle)$ and $x_{\alpha_{i+1}} \in V(D\langle F_{i} \rangle)\cap V(D\langle F_{i+1} \rangle )$, such that there exists an $x_{\alpha_{i}}x_{\alpha_{i+1}}$-path in $D\langle F_{i} \rangle$, say $T_{i}$. 
Notice that $T_{i}$ is an $H$-path for every $i \in \{ 0, \ldots , r-1 \}$, and $x_{\alpha _{r}} \in V(D')$. Moreover, since $C$ is a path in $C_{\mathscr{F}}(D)$, then $A(T_{i}) \cap A(T_{j}) = \emptyset$ whenever $i \neq j$.

Now we consider  $C=\cup_{i=0}^{r-1} T_{i}$, and suppose that $C=(z_{1}, \ldots , z_{n})$. Notice that $z_{n} \in V(D ' )$ (because $z_{n}=x_{\alpha_{r}}$). 

\begin{description}
\item \textbf{Claim 1.} $l_{H}(C) \leq l+1$.

For every $i \in \{ 0, \ldots , r - 2 \}$, let $U_{i}=\{ l \in \{ 1, \ldots , n-1 \} : (z_{l-1}, z_{l}) \in F_{i} \}$ be, and $L = \{ i \in \{ 0, \ldots , r-2 \} : U_{i} \neq \emptyset \}$. For every $i \in L$, define $\beta_{i} = max U_{i}$. 
 
We will show that  $O_{H}(C) \subseteq \{ \beta_{i} : i \in L \}$.  If $m \in O_{H}(C)$, then it follows that $(\rho (z_{m-1}, z_{m}), \rho (z_{m}, z_{m+1}) ) \notin A(H)$. On the other hand, we have that $(z_{m-1}, z_{m} ) \in A(T_{j})$ for some $j \in \{ 0, \ldots , r-1 \}$.  Since $T_{j}$ is an $H$-path, then $(z_{m}, z_{m+1}) \notin A(T_{j})$, which implies that $m = max U_{j}$ and $j \leq r-2$. Hence, $m \in \{\beta_{i} : i \in L \}$. It follows that $O_{H}(C) \subseteq L$. Hence, $|O_{H}(C)| \leq |L|$, that is $|O_{H}(C)| \leq r-1$, and we can conclude that $l_{H}(C) \leq l$.
\end{description} 
 
If $z_{n} \in K$, then by Claim 1 we have that $C$ is an $x_{0}K$-path with $l_{H}(C) \leq l +1$. If $z_{t} \notin K$, since $z_{n} \in V(D')$, and $K$ is an absorbent set by paths in $D'$, then there exists a $z_{n}K$-path in $D'$, say $T_{r}$. Hence, $C' = C \cup T_{r}$ is an $x_{0}K-$walk in $D$, and by Lemma \ref{obs1}, $l_{H}(C') \leq l_{H}(C) + 1$, which implies that $l_{H}(C') \leq l + 1$. 

Therefore, $K$ is an $(l+1, H)$-absorbent set by walks in $D$.
\end{proof}

Notice that conclusion of Proposition \ref{c1.p2} is tight. We will show an example where $K$ is not necessarily an $(r,H)$-absorbent set by walks in $D$ for some $r\leq l+1$. Consider the $H$-colored digraph shown in Figure \ref{figure1}, and for every $i \in \{1, \ldots , 6 \}$ let $F_{i}=\{ e \in A(D): \rho (e)= c_{i} \}$. Clearly, $\mathscr{F}=\{ F_{i} : i \in \{ 1, \ldots , 6 \}\}$ is a walk-preservative $H$-class partition of $A(D)$. Notice that $\mathcal{S}=\{F_{6}\}$ is a $3$-absorbent set in $C_{\mathscr{F}}(D)$. On the other hand, it is straightforward to see that $K=\{ x_{4} \}$ is a kernel by paths in $D\langle F_{6} \rangle$ which is not an $(r,H)$-absorbent by walks in $D$ for every $r \in \{ 1,2,3 \}$.

\begin{figure}[ht]
\centering
\includegraphics[scale=0.5]{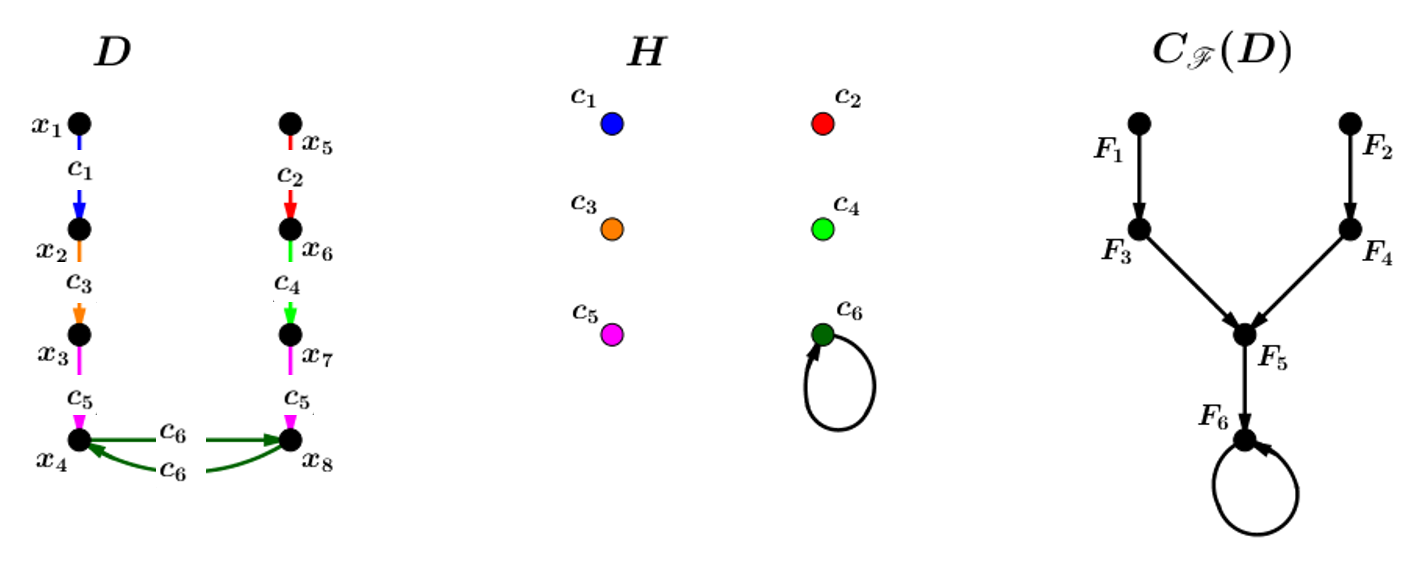}
\caption{\label{figure1}}
\end{figure}

On the other hand, if $D$ is an $H$-colored digraph and $\mathscr{F}$ is an $H$-class partition which is not walk-preservative, then Proposition \ref{c1.p2} is not necessarily true. Consider the $H$-colored digraph shown in Figure \ref{fig2}. For every $i \in \{1, \ldots , 5 \}$ let $F_{i}=\{ e \in A(D): \rho (e)=c_{i} \}$. It is straightforward to see that $\mathscr{F}=\{ F_{i} : i \in \{ 1, \ldots , 5\}\}$ is an $H$-class partition of $A(D)$. Notice that $(F_{2}, F_{3})$ is an arc of $C_{\mathscr{F}}(D)$ such that $x_{3} \in V(D\langle F_{2} \rangle)$, and there is no $x_{3}w$-walk in $D\langle F_{2} \rangle$ with $w \in V(D \langle F_{3} \rangle )$, that is, $\mathscr{F}$ is not a walk-preservative $H$-class partition of $A(D)$. On the other hand, $\mathcal{S}=\{F_{5} \}$ is a $4$-absorbent set in $C_{\mathscr{F}}(D)$ but no kernel by paths in $D\langle F_{5} \rangle$ is a $(5, H)$-absorbent set by walks in $D$.

\begin{figure}[ht]
\centering
\includegraphics[scale=0.5]{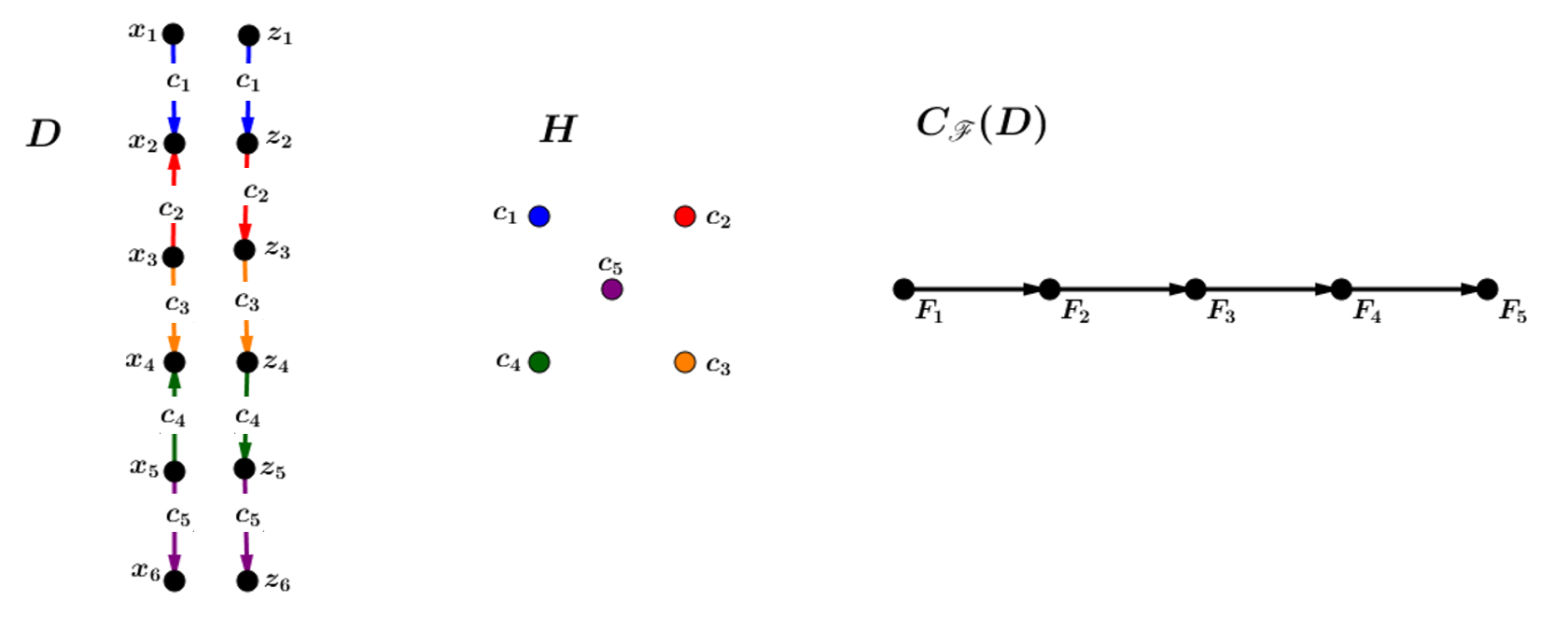}
\caption{\label{fig2}}
\end{figure}

Given an $H$-colored digraph $D$, a walk-preservative $H$-class partition of $A(D)$, say $\mathscr{F}$, and an $l$-absorbent an independent set in $C_{\mathscr{F}}(D)$,  it follows from Proposition \ref{c1.p2} that every kernel by paths in $D\langle \cup _{F \in N } F \rangle$, say $K$, is an $(l+1,H)$-absorbent set by walks in $D$. Remark that find kernels by paths in arbitrary digraphs can be solved in polynomial time.

\begin{lemma}
\label{c1.l0}
Let $D$ be an $H$-colored digraph and $\mathscr{F}$ an $H$-class partition of $A(D)$. If $\mathcal{S}$ is a nonempty subset of $\mathscr{F}$ and $K$ is a kernel by paths in $D\langle \cup _{F \in \mathcal{S}} F \rangle$, then  for every $x \in K$, $N^{-}_{\mathscr{F}}(x) \cap \mathcal{S} \neq \emptyset$.
\end{lemma}
\begin{proof}
Let $K$ be a kernel by paths in $D'=D\langle \cup _{F \in \mathcal{S}} F \rangle$ and $x \in K$. By Lemma \ref{c0.l1} we have that $d^{-}_{D'}(x) \neq 0$. Hence, there exists $z \in V(D')$ such that $(z,x) \in A(D')$. It follows from the definition of $D'$ that $(z,x) \in F$ for some $F \in \mathcal{S}$, which implies that $F\in N^{-}_{\mathscr{F}}(x) \cap \mathcal{S}$. Hence, $N^{-}_{\mathscr{F}}(x) \cap \mathcal{S} \neq \emptyset$.
\end{proof}

\begin{prop}
\label{c1.p1bis}
Let $D$ be an $H$-colored digraph and $\mathscr{F}$ a non-trivial walk-preservative $H$-class partition of $A(D)$ such that $C_{\mathscr{F}}(D)$ has no sinks. If $\mathcal{S}$ is an independent set in $C_{\mathscr{F}}(D)$, then there exists a kernel by paths in $D\langle \cup_{F \in \mathcal{S}} F\rangle$, say $N$, such that $N \subseteq V(D\langle \cup_{G \in  N^{+}(\mathcal{S})} G \rangle )$.
\end{prop}
\begin{proof}
We will denote by $D_{1}$ the digraph $D\langle \cup_{F \in \mathcal{S}} F\rangle$. Since $C_{\mathscr{F}}(D)$ has no sinks, then $N^{+}(\mathcal{S})\neq \emptyset$. Hence, denote by $D_{2}$ the digraph $D\langle \cup_{G \in  N^{+}(\mathcal{S})} G \rangle$. Let $N$ be a kernel by paths in $D_{1}$ intersecting $V(D_{2})$ the most possible, that is, for every kernel by paths in $D_{1}$, say $N'$, we have that
	\begin{equation}
	\label{rb.l3.ec1}
	|N \setminus V(D_{2})| \leq |N' \setminus V(D_{2})|
	\end{equation}
Notice that possibly $N \cap V(D_{2})=\emptyset$. We claim that $N \setminus V(D_{2})=\emptyset$. Proceeding by contradiction, suppose that there exists $x_{0} \in N \setminus V(D_{2})$.  Since $x_{0} \in V(D_{1})$, it follows from the definition of $D_{1}$ that there exists $F \in \mathcal{S}$ such that $x_{0} \in V(D \langle F \rangle)$. By hypothesis, $C_{\mathscr{F}}(D)$ has no sinks, which implies that there exists $G \in V(C_{\mathscr{F}}(D))$ such that $F \neq G$ and $(F, G) \in A(C_{\mathscr{F}}(D))$. It follows from the fact that $\mathscr{F}$ is a walk-preservative $H$-class partition that there exists an $x_{0}z$-path in $D\langle F \rangle$, say $C$, for some $z \in V(D \langle G \rangle)$. Notice that $C$ is a path in $D_{1}$, $z \in V(D_{1}) \cap V(D_{2})$, and $z \neq x_{0}$. Moreover, since $N$ is independent by paths in $D_{1}$, then $z \notin N$. We will prove the following claims in order to get a contradiction.
\begin{description}
	\item \textbf{Claim 1.} There exists a $zx_{0}$-path in $D_{1}$.

Since $z \notin N$, there exists a $zy$-path in $D_{1}$ for some $y \in N$, say $P'$. It follows that $C \cup P'$ is an $x_{0}y$-walk in $D_{1}$ such that $x_{0} \in N$ and $y\in N$, which implies that $x_{0}=y$, concluding that $P'$ is a $zx_{0}$-path in $D_{1}$.

	\item \textbf{Claim 2.} $(N \setminus \{ x_{0}\}) \cup \{ z \}$ is a kernel by paths in $D_{1}$. 

In order to show that $N'=(N \setminus \{ x_{0}\}) \cup \{ z \}$ is an absorbent set by paths in $D_{1}$, consider $u \in V(D_{1}) \setminus N'$. 
If $u=x_{0}$, then $C$ is an $x_{0}N'$-path in $D_{1}$. 
If $u \neq x_{0}$, then $u \in V(D_{1}) \setminus N$. It follows from the fact that $N$ is an absorbent set by paths in $D_{1}$ that there exists a $uv$-path in $D_{1}$, say $P$, for some $v \in N$. If $v \neq x_{0}$, then $P$ is a $uN'$-path in $D_{1}$. If $v=x_{0}$, then $P \cup C$ is a $uN'$-walk in $D_{1}$, concluding that $N'$ is an absorbent set by paths in $D_{1}$.

Now we will show that $N'$ is an independent set by paths in $D_{1}$. Proceeding by contradiction, suppose that there exists a $uv$-path in $D_{1}$, say $T$, where $\{ u, v \} \subseteq N'$. 
Since $N$ is an independent set by paths in $D_{1}$, then $z \in \{ u, v \}$. 
If $z=u$, it follows that  $C \cup T$ is an $x_{0}v$-walk in $D_{1}$, which contradicts the independence by paths of $N$. If $v=z$, by Claim 1 there exists a $zx_{0}$-path in $D_{1}$, say $P$, which implies that $T \cup P$ is a $ux_{0}$-walk in $D_{1}$, contradicting the independence by paths of $N$. Therefore, $N'$ is an independent set by paths in $D_{1}$, and the claim holds.
\end{description}

Notice that $| N' \setminus V(D_{2})| = | N \setminus V(D_{2})|-1$, which is not possible by (1). Therefore, $N \setminus V(D_{2})=\emptyset$, concluding that $N \subseteq  V(D_{2})$. 
\end{proof}

Notice that the previous proof gives us a simple way to find the kernel by paths described in Proposition \ref{c1.p1bis}. First, let $D_{2} = D\langle \cup_{G \in  N^{+}(\mathcal{S})} G \rangle$, and take an arbitrary kernel by walks in $D_{1}= D\langle \cup_{F \in \mathcal{S}} F\rangle$, say $N$. If $N \subseteq V(D_{2})$, we are done. Otherwise, there exist $x_{0} \in N \setminus V(D_{2})$, $z \in V(D_{1}) \cap V(D_{2})$ and a $x_{0}z$-path in $D_{1}$. Then, replace $N$ by $(N \setminus \{ x_{0} \}) \cup \{z\}$. Repeat such procedure until $N \subseteq V(D_{2})$.

\section{Main results} 

In this section we will show some conditions that guarantee the existence of $(k,l,H)$-kernels by walks in $H$-colored digraphs by means of $(k,l)$-kernels in the $H$-class digraph.

\begin{prop}
\label{propgird2}
Let $D$ be an $H$-colored digraph, and $\mathscr{F}$ a walk-preservative $H$-class partition of $A(D)$ such that $C_{\mathscr{F}}(D)$ has a $(k,l)$-kernel, say $\mathcal{S}$. If the following conditions hold:
	\begin{enumerate}[a)]
	\item $C_{\mathscr{F}}(D)$ has no sinks, and every cycle in $C_{\mathscr{F}}(D)$ is either a loop or has length at least $k$.
	
	\item For every $x \in V(D)$ such that $N_{\mathscr{F}}(x) \cap S \neq \emptyset $ and $N_{\mathscr{F}}(x) \cap N^{+}(\mathcal{S}) \neq \emptyset$, we have that $N^{-}_{\mathscr{F}}(x) \subseteq \mathcal{S}$.
	\end{enumerate}
Then $D$ has a $(k,l+1,H)$-kernel by walks.
\end{prop}
\begin{proof} 
First, suppose that $D$ has not isolated vertices. Let $D_{1}= D\langle \cup_{F\in \mathcal{S}} F \rangle$.
Since $C_{\mathscr{F}}(D)$ has no sinks and $\mathcal{S}$ is an independent set in $C_{\mathscr{F}}(D)$, then $N^{+}(\mathcal{S}) \neq \emptyset$. Now, let $D_{2}= D\langle \cup _{G\in N^{+}(\mathcal{S})}G \rangle$.  By Proposition \ref{c1.p1bis}, we can consider a kernel by paths in $D_{1}$, say $K$, such that $K \subseteq V(D_{2})$. We will show that $K$ is a $(k,l+1,H)$-kernel by walks in $D$.

Since $\mathscr{F}$ is walk-preservative, it follows from Proposition \ref{c1.p2} that $K$ is an $(l+1,H)$-absorbent set by walks in $D$. 
It only remains to show that $K$ is a $(k, H)$-independent set by walks in $D$. First, we will prove the following useful claim.

\begin{description}
\item \textbf{Claim 1.} For every $x \in K$, $N^{-}_{\mathscr{F}}(x) \subseteq A(D_{1})$.

If $x \in K$, then $x \in V(D_{1}) \cap V(D_{2})$. Since $x \in V(D_{1})$, it follows from the definition of $D_{1}$ that there exists $F \in \mathcal{S}$ such that $A(x) \cap F \neq \emptyset$. Hence, $N_{\mathscr{F}}(x) \cap  \mathcal{S} \neq \emptyset$. On the other hand, since $x \in V(D_{2})$, an analogous proof will show that $N_{\mathscr{F}}(x) \cap N^{+}(\mathcal{S} )\neq \emptyset$. 
By hypothesis (b), we can conclude that $N^{-}_{\mathscr{F}}(x) \subseteq \mathcal{S}$.
\end{description}

In order to show that $K$ is a $(k,H)-$independent set by walks in $D$, consider an $x_{0}x_{n}$-walk in $D$, say $T=(x_{0}, \ldots , x_{n})$, such that $\{ x_{0}, x_{n} \} \subseteq K$. 
 
	\begin{description}
	\item \textbf{Claim 2.} $O_{H}(T) \neq \emptyset$.
	
	 Proceeding by contradiction, suppose that $O_{H}(T)=\emptyset$. Hence, there exists $F' \in \mathscr{F}$ such that $A(T) \subseteq F'$, which implies that $(x_{n-1}, x_{n}) \in F'$. It follows from Claim 1 that $F' \in \mathcal{S}$. Hence, $T$ is an $x_{0}x_{n}$-walk in $D_{1}$, which contradicts the fact that $K$ is an independent set by paths in $D_{1}$. Therefore, $O_{H}(T)\neq \emptyset$ and the claim holds.	
	\end{description}

By Claim 2, suppose that $O_{H}(T)=\{ \alpha_{i} : i \in \{1, \ldots , t \} \}$ where $t \geq 1$, and $\alpha_{i} \leq \alpha_{i+1}$ for every $ i \in \{ 1, \ldots , t-1 \}$. On the other hand, for every $i \in \{ 1, \ldots , t \}$ consider $F_{i} \in \mathscr{F}$ such that $(x_{\alpha_{i}}, x_{\alpha_{i}}^{+} ) \in F_{i}$, and $F_{0} \in \mathscr{F}$ such that $(x_{0}, x_{1}) \in F_{0}$. It follows from the definition of $C_{\mathscr{F}}(D)$ that $T'=(F_{0}, F_{1}, \ldots , F_{t})$ is a walk in $C_{\mathscr{F}}(D)$. Notice that $l_{H}(T)=l(T')+1$ and, by Claim 1, $F_{t} \in \mathcal{S}$. Consider the following cases:

	\begin{description}
	\item	\textbf{Case 1.} $F_{0} \in \mathcal{S}$.

If $F_{0} \neq F_{t}$, as $\mathcal{S}$ is a $k$-independent set in $C_{\mathscr{F}}(D)$, then $l(T') \geq k$, which implies that $l_{H}(T) \geq k$.  
If $F_{0} =F_{t}$, then $T'$ is a closed walk in $C_{\mathscr{F}}(D)$ which is not a loop and, by hypothesis, $l(T')\geq k$. Hence, $l_{H}(T) \geq k$.
 
	\item \textbf{Case 2.} $F_{0} \in V(C_{\mathscr{F}}(D))\setminus \mathcal{S}$.

By Lemma \ref{c1.l0}, consider $F \in N^{-}_{\mathscr{F}}(x_{0}) \cap \mathcal{S}$. Since $F_{0} \notin \mathcal{S}$, then $F \neq F_{0}$. It follows from the definition of $C_{\mathscr{F}}(D)$ that $T'' =(F, F_{0} ) \cup T'$ is a walk in $C_{\mathscr{F}}(D)$. Notice that $l_{H}(T)=l(T'')$. If $F\neq F_{t}$, as  $\mathcal{S}$ is a $k$-independent set in $C_{\mathscr{F}}(D)$, we have that $l(T'') \geq k$, which implies that $l_{H}(T) \geq k$. If $F=F_{t}$, then  $T''$ is a closed walk in $C_{\mathscr{F}}(D)$ which is not a loop and, by hypothesis, $l(T'')\geq k$. Hence, $l_{H}(T) \geq k$.
	\end{description}

It follows from Case 1 and Case 2 that $K$ is a $(k,H)$-independent set by walks in $D$. Therefore, $K$ is a $(k,l+1)$-kernel by walks in $D$. 

Now, suppose that $W=\{ x \in V(D) : d(x)=0 \}$ is nonempty. By the previous proof, we have that $D - W$ has a $(k,l +1 ,H)$-kernel by walks, say $K$. By Lemma \ref{noisolatedvertices}, we can conclude that $K \cup W$ is a $(k, l+1,H)$-kernel by walks in $D$.
\end{proof}

Remark that, given an $H$-colored digraph that satisfies the hypothesis of proposition \ref{propgird2}, every kernel by walks $K$ in $D\langle \cup _{F \in \mathcal{S}} F \rangle $ contained in $D\langle \cup _{G\in N^{+}(\mathcal{S})}G \rangle$ is a $(k,l+1,H)$-kernel by walks in $D$. Such kind of kernel by walks can be found in polynomial time, as we show previously.

\begin{prop}
\label{emptyoutneighborhood}
Let $D$ be an $H$-colored digraph with not isolated vertices, $\mathscr{F}$ a walk-preservative $H$-class partition of $A(D)$, and $\mathcal{S}\subseteq V(C_{\mathscr{F}}(D))$ an independent and $l$-absorbent set in $C_{\mathscr{F}}(D)$ for some $l \geq 1$. If $N^{+}(\mathcal{S}) = \emptyset$ and $k \geq 2$, then every kernel by paths in $D\langle \cup_{F\in \mathcal{S}} F \rangle$ is a $(k,l+1,H)$-kernel by walks in $D$.
\end{prop}
\begin{proof}
Let $K$ be a kernel by paths in $D'=D\langle \cup _{F\in \mathcal{S}} F \rangle$. Since $\mathscr{F}$ is walk preservative, it follows from  Proposition \ref{c1.p2} that $K$ is an $(l+1,H)$-absorbent set by walks in $D$. Now, in order to prove that $K$ is a $(k,H)$-independent set by walks in $D$ for every $k \geq 2$, we will show that $K$ is a path-independent set in $D$. Proceeding by contradiction, suppose that there exists an $x_{1}x_{n}$-path in $D$, say $T'=(x_{1}, \ldots , x_{n})$, such that $\{ x_{1}, x_{n} \} \subseteq K$. Consider $F_{0} \in N^{-}_{\mathscr{F}}(x_{1})\cap \mathcal{S}$ (Lemma \ref{c1.l0}). Hence, there exists $x_{0} \in V(D)$ such that $(x_{0}, x_{1}) \in F_{0} \cap A^{-}(x_{1})$, and let $T=(x_{0}, x_{1}) \cup T'$ be.

If $A(T) \subseteq F_{0}$, then $A(T') \subseteq F_{0}$, which implies that $T'$ is an $x_{1}x_{n}$-path in $D'$, contradicting the fact that $K$ is an independent set by paths in $D'$. Hence, $A(T)\not \subseteq F_{0}$. Let $t = min \{ i \in \{ 1, \ldots , n-1 \} : (x_{i}, x_{i+1}) \notin F_{0} \}$. It follows that $(x_{t-1}, x_{t}) \in F_{0}$, and $(x_{t}, x_{t+1}) \in G$ for some $G \in \mathscr{F}$ with $F\neq G$. Notice that $(F,G) \in A(C_{\mathscr{F}}(D))$ and, since $\mathcal{S}$ is an independent set in $C_{\mathscr{F}}(D)$, then $G \notin \mathcal{S}$. Therefore, $G \in N^{+}(\mathcal{S})$ which contradicts the assumption that $N^{+}(\mathcal{S})=\emptyset$. Hence, $K$ is path-independent in $D$, which implies that $K$ is a $(k,H)$-independent set by walks in $D$ for every $k \geq 2$.

Therefore $K$ is a $(k,l+1,H)$-kernel by walks in $D$ for every $k \geq 2$.
\end{proof}

\begin{prop}
\label{unilateraltheorem}
Let $D$ be an $H$-colored digraph with not isolated vertices, $\mathscr{F}$ a walk-preservative $H$-class partition of $A(D)$, and $\mathcal{S}$ a $(k,l)$-kernel in $C_{\mathscr{F}}(D)$ such that $k \geq 3$ and $l \geq 1$. If for every $F \in \mathcal{S}$, $D\langle F \rangle$ is unilateral and has no sinks, 
 then every kernel by paths in $D\langle \cup _{F \in \mathcal{S}} F \rangle$ is a $(k-1,l+1,H)$-kernel by walks in $D$.
\end{prop}
\begin{proof}
Let $K$ be a kernel by paths in $D_{1}= D\langle \cup_{F\in \mathcal{S}} F \rangle$.  Since $\mathscr{F}$ is a walk-preservative $H$-class partition, it follows from Proposition \ref{c1.p2} that $K$ is an $(l+1,H)$-absorbent set by walks in $D$. It only remains to show that $K$ is a $(k-1, H)$-independent set by walks in $D$.  Consider a walk in $D$, say $C=(x_{0}, \ldots , x_{n})$, such that $\{ x_{0}, x_{n} \} \subseteq K$.
	\begin{description}
	\item \textbf{Claim 1.}  $N^{-}_{\mathscr{F}}(x_{0}) \cap \mathcal{S} \neq \emptyset$ and $N^{+}_{\mathscr{F}}(x_{n}) \cap \mathcal{S} \neq \emptyset$.

	Since $x_{0} \in K$,  it follows from Lemma \ref{c1.l0} that  $N^{-}_{\mathscr{F}}(x_{0}) \cap \mathcal{S} \neq \emptyset$. On the other hand, since $x_{n} \in V(D_{1})$, then $x_{n} \in V(D \langle F \rangle )$ for some $F \in \mathcal{S}$. By hypothesis, $D \langle F \rangle$ has no sinks, which implies that $A^{+}(x_{n}) \cap F \neq \emptyset$. Hence $F \in N^{+}_{\mathscr{F}}(x_{n})$, concluding that  $F \in N^{+}_{\mathscr{F}}(x_{n}) \cap \mathcal{S}$, and the claim holds.
	\end{description}
	
By Claim 1, consider $F' \in N^{-}_{\mathscr{F}}(x_{0}) \cap \mathcal{S}$ and $F'' \in N^{+}_{\mathscr{F}}(x_{n}) \cap \mathcal{S}$. It follows from Lemma \ref{unilateral.lema} that $F' \neq F''$ (*).
	
	\begin{description}
	\item \textbf{Claim 2.} $O_{H}(C) \neq \emptyset$.
	
	Proceeding by contradiction, suppose that $O_{H}(C)= \emptyset$. Hence, there exists $F \in \mathscr{F}$ such that $A(C)\subseteq F$. By definition of $C_{\mathscr{F}}(D)$, we have that $\{ (F', F) , (F, F'') \} \subseteq A(C_{\mathscr{F}}(D))$. Since $\{ F', F'' \} \subseteq \mathcal{S}$ and $\mathcal{S}$ is an independent set in $C_{\mathscr{F}}(D)$, then $F \neq F'$ and $F \neq F''$. Hence, $(F', F, F'')$ is an $F'F''$-path in $C_{\mathscr{F}}(D)$, which is not possible since $\mathcal{S}$ is a $k$-independent set with $k \geq 3$. Therefore, $O_{H}(C) \neq \emptyset$, and the claim holds. 
	\end{description}

By Claim 2, suppose that $O_{H}(C)=\{ \alpha_{i} : i \in \{1, \ldots , t \} \}$ where $t \geq 1$, and $\alpha_{i} \leq \alpha_{i+1}$ for every $i \in  \{ 1, \ldots , t-1 \}$. For every $i \in \{ 1, \ldots , t \}$, let $F_{i} \in \mathscr{F}$ such that $(x_{\alpha_{i}}, x_{\alpha_{i}}^{+} ) \in F_{i}$, and $F_{0} \in \mathscr{F}$ such that $(x_{0}, x_{1}) \in F_{0}$. 
Notice that $x_{0} \in V(D \langle F_{0} \rangle)$ and $x_{n} \in V(D \langle F_{t} \rangle )$. On the other hand, it follows from the definition of $C_{\mathscr{F}}(D)$ that $C_{0}=(F_{0}, F_{1}, \ldots , F_{t})$ is a walk in $C_{\mathscr{F}}(D)$.  Consider the following cases:

\begin{description}
	\item \textbf{Case 1.} $F_{0} \in \mathcal{S}$.
	
		First, suppose that $F_{t} \in \mathcal{S}$. Notice that $C_{0}$ is a walk in $C_{\mathscr{F}}(D)$ such that $l_{H}(C) - 1=l(C_{0})$. On the other hand, we have that $\{ F_{0}, F_{t} \} \subseteq \mathcal{S}$, $x_{0} \in V(D \langle F_{0} \rangle )$ and $x_{n} \in V(D\langle F_{t} \rangle )$, which implies that $F_{0} \neq F_{t}$ (Lemma \ref{unilateral.lema}). Since $\mathcal{S}$ is a $k$-independent set in $C_{\mathscr{F}}(D)$, we have that $l(C_{0}) \geq k$. We can conclude that $l_{H}(C) \geq k-1$. 	
		
		Now suppose that $F_{t} \notin \mathcal{S}$, and consider $C_{1}=C_{0} \cup (F_{t}, F'')$. Notice that $C_{1}$ is a walk in $C_{\mathscr{F}}(D)$ such that $l_{H}(C)=l(C_{1})$. On the other hand, we have that $\{ F_{0}, F'' \} \subseteq \mathcal{S}$, $x_{0} \in V(D \langle F_{0} \rangle )$ and $x_{n} \in V(D\langle F'' \rangle )$, which implies that $F_{0} \neq F''$ (Lemma \ref{unilateral.lema}).  Since $\mathcal{S}$ is a $k$-independent set in $C_{\mathscr{F}}(D)$, then $l(C_{1}) \geq k$. We can conclude that $l_{H}(C)\geq k-1$.
			
	\item \textbf{Case 2.} $F_{0} \in V(C_{\mathscr{F}}(D))\setminus \mathcal{S}$.
	
			First, suppose that $F_{t} \in \mathcal{S}$, and consider $C_{2}=(F', F_{0} ) \cup C_{0}$. Notice that $C_{2}$ is a walk in $C_{\mathscr{F}}(D)$ such that $l_{H}(C)=l(C_{2})$. On the other hand, we have that $\{ F', F_{t} \} \subseteq \mathcal{S}$, $x_{0} \in V(D \langle F' \rangle )$ and $x_{n} \in V(D\langle F_{t} \rangle )$, which implies that $F' \neq F_{t}$ (Lemma \ref{unilateral.lema}). Since $\mathcal{S}$ is a $k$-independent set in $C_{\mathscr{F}}(D)$, then $l(C_{2}) \geq k$. We can conclude that $l_{H}(C)\geq k-1$.
			
			Now suppose that $F_{t} \notin \mathcal{S}$, and consider $C_{3}=(F', F_{0} ) \cup C_{0} \cup (F_{t}, F'')$. Notice that $C_{3}$ is a walk in $C_{\mathscr{F}}(D)$ such that $l_{H}(C) + 1=l(C_{3})$. On the other hand, by (*) we have that $F' \neq F''$ and, since $\mathcal{S}$ is a $k$-independent set in $C_{\mathscr{F}}(D)$, then $l(C_{3}) \geq k$. We can conclude that $l_{H}(C)\geq k-1$.
\end{description}

It follows from the previous cases that $K$ is a $(k-1,H)$-independent set by walks in $D$. Therefore, $K$ is a $(k-1,l+1)$-kernel by walks in $D$. 
\end{proof}

Remark that, given an $H$-colored digraph that satisfies the hypothesis of proposition \ref{emptyoutneighborhood}, every kernel by walks in $D\langle \cup _{F \in \mathcal{S}} F \rangle $ is a $(k,l+1,H)$-kernel by walks in $D$. In the same way, for every $H$-colored digraph that satisfies the hypothesis of proposition \ref{unilateraltheorem}, every kernel by walks in $D\langle \cup _{F \in \mathcal{S}} F \rangle $ is a $(k-1,l+1,H)$-kernel by walks in $D$

\begin{prop}
\label{propstrong2}
Let $D$ be an $H$-colored digraph, $\mathscr{F}$ an $H$-class partition of $A(D)$ and  $\mathcal{S}$ a $(k,l)$-kernel of $C_{\mathscr{F}}(D)$ for some $k \geq 3$ and $l \geq 1$. If for every $F \in \mathcal{S}$, $D \langle F \rangle$ is strongly connected and has an obstruction-free vertex in $D$, then $D$ has a $(k+1,l+1,H)$-kernel by walks.
\end{prop}
\begin{proof}
First, suppose that $D$ has not isolated vertices. Let $\mathcal{S}=\{ F_{1}, \ldots , F_{r} \}$ for some $r \geq 1$, and for every $i \in \{1, \ldots , r\}$, let $z_{i} \in V(D \langle F_{i} \rangle )$ such that $z_{i}$ is obstruction-free in $D$. By Lemma \ref{lemmastronglyconnected} (c) we have that $z_{i} \neq z_{j}$ whenever $\{ i, j \} \subseteq \{ 1, \ldots , r \}$ and $i \neq j$ .

\begin{description}
	\item \textbf{Claim 1.} $K=\{ z_{i} : i \in \{ 1, \ldots , r \}\}$ is a kernel by paths in $D_{1}=D \langle \cup _{F\in \mathcal{S}} F \rangle$.
	
	 In order to show that $K$ is an absorbent set by paths in $D_{1}$, consider $w \in V(D_{1}) \setminus K$. Since $w \in V(D_{1})$, then there exists $j \in \{ 1, \ldots , r \}$ such that $w \in V(D \langle F_{j} \rangle)$. It follows from the fact that $D\langle F_{j} \rangle$ is strongly connected that there exists a $wz_{j}$-path in $D \langle F_{j} \rangle$, say $P$. Hence, $P$ is a $wK$-path in $D_{1}$, concluding that $K$ is an absorbent set by paths in $D_{1}$.

It only remains to show that $K$ is a path-independent set in $D_{1}$. It follows from Lemma \ref{lemmastronglyconnected} (c) that $V(D\langle F_{i} \rangle ) \cap V(D \langle F_{j} \rangle ) = \emptyset$ for every $\{ i, j \} \subseteq \{ 1, \ldots r \}$ with $i \neq j$, which implies that there is no $z_{i}z_{j}$-path in $D_{1}$ for every $\{ i, j \} \subseteq \{ 1, \ldots r \}$ with $i \neq j$. Hence, $K$ is a path-independent set in $D_{1}$, and the claim holds.
\end{description}

Now, we will show that $K$ is a $(k+1, l+1, H)$-kernel by walks in $D$. In order to show that $K$ is an $(l+1,H)$-absorbent set by walks in $D$, notice that $\mathscr{F}$ is a walk-preservative partition of $A(D)$ (Lemma \ref{lemmastronglyconnected} (a)), and, since $K$ is a kernel by walks in $D_{1}$, we can conclude from Proposition \ref{c1.p2} that $K$ is an $(l+1,H)$-absorbent set by walks in $D$.

 It only remains to show that $K$ is a $(k+1, H)$-independent set by walks in $D$. Consider $\{ z_{i}, z_{j} \} \subseteq K$ with $i \neq j$, and a  $z_{i}z_{j}$-walk in $D$, say $C=(z_{i}=x_{0}, x_{1} \ldots , x_{n} = z_{j})$. 

\begin{description}
\item \textbf{Claim 2.} $O_{H}(C) \neq \emptyset$.

Proceeding by contradiction, suppose that $O_{H}(C) = \emptyset$, which implies that there exists $F \in \mathscr{F}$ such that $A(C)\subseteq F$. Since $z_{i}$ and $z_{j}$ are obstruction-free in $D$, then $F = F_{i}$ and $F = F_{j}$ (Lemma \ref{obs1} (b)), concluding that $F_{i}=F_{j}$, which is not possible since $F_{i} \neq F_{j}$. Hence, $O_{H}(C) \neq \emptyset$ and the claim holds.
\end{description}

By Claim 2, suppose that $O_{H}(C)=\{ \alpha_{i} : i \in \{1, \ldots , t \} \}$ where $t \geq 1$, and $\alpha_{i} \leq \alpha_{i+1}$ for every $ i \in \{ 1, \ldots , t-1 \}$. For every $i \in \{ 1, \ldots , t \}$, let $G_{i} \in \mathscr{F}$ such that $(x_{\alpha_{i}}, x_{\alpha_{i}}^{+} ) \in G_{i}$, and $G_{0} \in \mathscr{F}$ such that $(x_{0}, x_{1}) \in G_{0}$.
 It follows from the definition of $C_{\mathscr{F}}(D)$ that $C'=(G_{0}, G_{1}, \ldots , G_{t})$ is a walk in $C_{\mathscr{F}}(D)$. Notice that $z_{i} \in V(D\langle G_{0} \rangle )$, $z_{j} \in V(D\langle G_{t} \rangle )$, and $l_{H}(C)= l(C') +1$. Since $z_{i}$ and $z_{j}$ are obstruction-free in $D$, then $F_{i} =G_{0}$ and $F_{j} = G_{t}$ (Lemma \ref{obs1} (b)), which implies that $\{ G_{0}, G_{t} \} \subseteq \mathcal{S}$ and $G_{0} \neq G_{t}$. Since $\mathcal{S}$ is a $k$-independent set in $C_{\mathscr{F}}(D)$, then $l(C') \geq k $, which implies that $ l_{H}(C) \geq k+1$. Hence, $K$ is a $(k+1, H)$-independent set by walks in $D$. Therefore, $K$ is a $(k+1, l+1, H)$-kernel by walks in $D$.

On the other hand, if $W=\{x\in V(D) : d(x) =0\}$ is nonempty, then by the previous proof, we have that $D - W$ has a $(k+1,l +1 ,H)$-kernel by walks, say $K$. By Lemma \ref{noisolatedvertices}, we can conclude that $K \cup W$ is a  $(k+1,l+1, H)$-kernel by walks in $D$.
\end{proof}

In the same spirit that in the previous results, given an $H$-colored digraph that satisfies the hypothesis of Proposition \ref{propstrong2}, it is easy to find a $(k+1, l+1,H)$-kernel by walks, by taking the obstruction-free vertex in $D\langle F \rangle$ for every $F \in \mathcal{S}$.

\section{Some consequences}

Finally, we present the following theorems which are direct consequences of the results presented in the previous section. 

\begin{theorem}
Let $D$ be an $H$-colored digraph, $\mathscr{F}$ a walk-preservative $H$-class partition of $A(D)$, and $\{k, l \} \subseteq \mathbb{N}$ such that $k \geq 2$ and $l \geq 1$. If $\mathcal{S}$ is a $(k,l)$-kernel in $C_{\mathscr{F}}(D)$ such that $N^{+}(\mathcal{S}) = \emptyset$, then $D$ has a $(k,l+1,H)$-kernel by walks.
\end{theorem}
\begin{proof}
Let $W=\{ x \in V(D) : d(x)= 0 \}$ and $\mathcal{S}$ a $(k,l)$-kernel in $C_{\mathscr{F}}(D)$ such that $N^{+}(\mathcal{S}) = \emptyset$. Since $\mathcal{S}$ is an independent and $l$-absorbent set in $C_{\mathscr{F}}(D)$, it follows from Proposition \ref{emptyoutneighborhood} that $D-W$ has a $(k,l,H)$-kernel by walks, say $K$. By Lemma \ref{noisolatedvertices}, we can conclude that $K \cup W$ is a $(k,l,H)$-kernel by walks in $D$.
\end{proof}

\begin{theorem}  
\label{propgird}
Let $D$ be a strongly connected $H$-colored digraph and $\mathscr{F}$ a walk-preservative $H$-class partition of $A(D)$ such that $C_{\mathscr{F}}(D)$ has a $(k,l)$-kernel, say $\mathcal{S}$. If the following conditions hold:
	\begin{enumerate}[a)]
	\item Every cycle in $C_{\mathscr{F}}(D)$ is either a loop or has length at least $k$.
	
	\item For every $x \in V(D)$ such that $N_{\mathscr{F}}(x) \cap S \neq \emptyset $ and $N_{\mathscr{F}}(x) \cap N^{+}(\mathcal{S}) \neq \emptyset$, we have that $N^{-}_{\mathscr{F}}(x) \subseteq \mathcal{S}$.
	\end{enumerate}
Then $D$ has a $(k,l+1,H)$-kernel by walks.
\end{theorem}
\begin{proof}
If $D$ is an $H$-digraph, it follows from Lemma \ref{hdigraph} that  for every $k \geq 2$ and $l \geq 1$, $D$ has a $(k,l,H)-$kernel by walks. Hence, we may assume that $D$ is not an $H$-digraph. It follows that $C_{\mathscr{F}}(D)$ is a non-trivial strongly connected digraph (Lemma \ref{c1.l2}), which implies that $C_{\mathscr{F}}(D)$ has no sinks. By Proposition \ref{propgird2} we have that $D$ has a $(k, l+1,H)$-kernel by walks.
\end{proof}

\begin{theorem}
\label{corunilat}
Let $D$ be an $H$-colored digraph, $\mathscr{F}$ a walk-preservative $H$-class partition of $A(D)$ such that for every $F \in \mathscr{F}$, $D\langle F \rangle$ is unilateral and has no sinks. If  $C_{\mathscr{F}}(D)$ has a $(k,l)$-kernel for some $k \geq 3$ and $l \geq 1$, then $D$ has a $(k-1,l+1,H)$-kernel by walks.
\end{theorem}
\begin{proof}
First, suppose that $D$ has not isolated vertices. If $\mathcal{S}$ is a $(k,l)$-kernel in $C_{\mathscr{F}}(D)$ with $k \geq 3$ and $l \geq 1$, then by Proposition \ref{unilateraltheorem}, every kernel by paths in $D \langle \cup _{F \in \mathcal{S}}F \rangle$ is a $(k-1,l+1,H)$-kernel by walks in $D$.

On the other hand, if $W=\{x\in V(D) : d(x) =0\}$ is nonempty, then by the previous proof, we have that $D - W$ has a $(k-1,l +1 ,H)$-kernel by walks, say $K$. By Lemma \ref{noisolatedvertices}, we can conclude that $K \cup W$ is a $(k-1, l+1,H)$-kernel by walks in $D$.
\end{proof}

\begin{theorem}
\label{Teostrong}
Let $D$ be an $H$-colored digraph and $\mathscr{F}$ an $H$-class partition of $A(D)$. If for every $F \in \mathscr{F}$ we have that $D\langle F \rangle $ is strongly connected, then for every $k \geq 2$ and $l \geq k+1$,  $D$ has a $(k,l, H)$-kernel by walks.
\end{theorem}
\begin{proof}
First, suppose that $D$ has not isolated vertices. Since $D\langle F \rangle $ is strongly connected for every $F \in \mathscr{F}$, we have that $\mathscr{F}$ is walk-preservative (Lemma \ref{lemmastronglyconnected} (a)),  and $D \langle F \rangle$ is unilateral and has no sinks  for every $F \in \mathscr{F}$. 
On the other hand, by Lemma \ref{lemmastronglyconnected} (b), we have that  $C_{\mathscr{F}}(D)$ is a symmetric digraph, which implies that $C_{\mathscr{F}}(D)$ has a $(k+1, l-1)$-kernel for every $k +1 \geq 3$ and $l -1 \geq k $ (Lemma \ref{symetricklkernel}). Hence, by Corollary  \ref{corunilat}, we have that $D$ has a $(k, l, H)$-kernel by walks for every $k \geq 2$ and $l \geq k+1$. 

On the other hand, if $W=\{x\in V(D) : d(x) =0\}$ is nonempty, then by the previous proof, we have that $D - W$ has a $(k,l ,H)$-kernel by walks for every $k \geq 2$ and $l \geq k +1$. By Lemma \ref{noisolatedvertices}, we can conclude that $D$ has a $(k,l, H)$-kernel by walks for every $k \geq 2$ and $l \geq k+1$.
\end{proof}

\begin{theorem}
\label{teostrong2}
Let $D$ be an $H$-colored digraph and $\mathscr{F}$ an $H$-class partition of $A(D)$ such that for every $F \in \mathscr{F}$, $D \langle F \rangle$ is strongly connected and has a obstruction-free vertex in $D$. For every $k \geq 2$, $D$ has a $(k,H)$-kernel by walks.
\end{theorem}
\begin{proof}
Since $C_{\mathscr{F}}(D)$ is a symmetric digraph (Lemma \ref{lemmastronglyconnected} (b)), for every $k \geq 2$ we have that $C_{\mathscr{F}}(D)$ has a $k$-kernel (Theorem \ref{c0.t1}). By Theorem \ref{propstrong2}, $D$ has a $(k, H)$-kernel by walks for every $k \geq 3$. On the other hand, it follows from Theorem \ref{classlema} that $D$ has a $(2,H)$-kernel by walks, concluding that $D$ has a $(k,H)$-kernel by walks for every $k \geq 2$.
\end{proof}

\section*{Acknowledgments}
Hortensia Galeana-Sánchez is supported by CONACYT FORDECYT-PRONACES/39570/2020 
and UNAM-DGAPA-PAPIIT IN102320. Miguel Tecpa-Galván is supported by
CONACYT-604315.

\end{document}